\documentclass[11pt,a4]{article}
\usepackage[linktocpage,
  pdfpagelabels=true,pdfstartview=Fit,bookmarksopen=true 
  ,pdfpagemode=UseOutlines,bookmarksnumbered=true,
  pdfauthor={Myron Benioudakis},
]{hyperref} 
\usepackage{bookmark}
\usepackage[a4paper,width=150mm,top=25mm,bottom=25mm,bindingoffset=6mm]{geometry}
\usepackage{amsmath}
\usepackage{amssymb} 
\usepackage{amsthm}
\usepackage[T1]{fontenc}
\usepackage{authblk} 
\usepackage{graphicx}  
\usepackage{epstopdf} 
\usepackage{fancyhdr}
\usepackage{mathtools}
\usepackage{natbib}
\title{Lead-Time Quotations in Unobservable Make-To-Order Systems with Strategic Customers: Risk Aversion, Load Control and Profit Maximization}

\author[1]{Myron Benioudakis}
\affil[1]{{\footnotesize  Department of Management Science and Technology\\
  Athens University of Economics and Business\\ 
  \url{benioudakis@aueb.gr}, \ \url{ioannou@aueb.gr}}}
  
\author[2]{Apostolos Burnetas}
\affil[2]{{\footnotesize  Department of Mathematics\\ National and Kapodistrian University of Athens\newline \url{aburnetas@math.uoa.gr}}}  

\author[1]{George Ioannou}

\date{\today}
\providecommand{\keywords}[1]{\textbf{\textit{Keywords:}} #1}
\newtheorem{assumption}{Assumption}
\newtheorem{mydef}{Definition} 

\newtheorem{proposition}{Proposition}
\newtheorem{theorem}{Theorem}

\begin{document}
\maketitle
\begin{abstract}
We develop a model for pricing, lead-time quotation and delay compensation in a Markovian make-to-order production or service system with strategic customers who exhibit risk aversion.
Based on a concave utility function of their net benefit, customers make individual decisions to join the system or balk without observing the state of the queue.  The decisions of arriving customers result in a symmetric join/balk game. Regarding the firm's strategy, the provider announces a fixed  entrance fee, a lead-time quotation and a compensation rate for the part of a customer delay which exceeds the quoted lead-time. 

We analyze the effect of customer risk aversion and the compensation policy on the equilibrium join/balk strategies and the resulting input rates, and assess the flexibility of the provider in inducing a range of possible input rates under various constraints on the pricing/compensation policy. In numerical experiments we explore the behavior of pricing curves that reflect the provider's choices in inducing specific input rates. A key insight obtained from the analysis is that a main benefit of the lead-time and compensation option is to allow the entrance fee to remain high and the provider prefers strategies that lead to this direction. 

\keywords {Queueing; Customer Equilibrium strategies; Load control; Profit Maximization; Risk Aversion.}
\end{abstract}

\section{Introduction}

In firms that produce products on a make-to-order basis or provide a service, customer delay is a crucial factor that affects both the demand and the profitability, directly or indirectly. Customers may react to delay, real or anticipated, in a variety of ways ranging from mild to very extreme, depending on the effect it has on their own operation. Therefore, production or service providers must be able to effectively address customers' concerns about delays in order for the demand not to be adversely affected. 

In considering the effects of delay on the customers and the firm, several interacting factors must be taken into account. First, in many cases customer behavior is strategic, in the sense that customers also consider the reaction of other customers when they make their decision to place the order for the product or service, i.e., whether to join the system or balk. This creates externalities, since the decision of a customer affects and is affected by the behavior of the other customers, and it may have unexpected consequences in the formation of demand patterns. In addition, customers are often affected disproportionately by long delays, since they may lead to defaults or extreme losses. In such cases customers exhibit a risk-averse behavior, which must also be considered as its effect on the demand may be quite adverse.

From the point of view of the firm, the producer or service provider, hereafter referred to as the provider, is mainly concerned about his/her own profitability. The presence of strategic customer behavior and in particular risk aversion generally induces constraints on the pricing strategies and the profit. Especially under risk aversion, single pricing schemes where customers pay a fixed fee to order the product or the service (entrance fee to the system) are not appropriate because they may lead to significant loss of customers. To alleviate these effects it is common to use delay-dependent pricing policies. A class of such policies, considered in this paper, is to provide customers with a lead-time quotation and compensate for delays exceeding the lead-time. The main role of a compensation strategy in this framework is that, by providing a hedge against extreme losses only to those customers who experience long delays, it allows the provider to keep the entrance fee at a relatively high level, instead of giving a steep discount to all incoming customers in order to entice them to join. 
The questions arising here have to do with the most appropriate combination of pricing and compensation. For example is it preferable to provide a modest compensation with a long lead-time and a lower entrance fee or keep the price high and compensate from the first moment? Furthermore, depending on the general market environment, the firm strategy and/or other factors, the provider may face limitations in setting the pricing/compensation policy, e.g. may be forced to keep the entrance fee at a fixed level, or use an already advertised lead-time quotation, etc. 

Finally, in terms of the firm's objectives, the standard approach is to determine the parameters of the pricing/compensation policy that maximize the provider's profit, and accept the demand resulting from this policy as optimal. However in many cases firms are interested in controlling the customer flow from a more general viewpoint, not necessarily tied to short-term profit maximization. The provider may want to determine a pricing/compensation policy that will shape a particular demand pattern over time, which may not be profit maximizing in the short-run but aim to increasing market share, product reputation etc. Such concerns and objectives are especially relevant for startup companies or new products introduced to the market. In such situations it is useful to know the policy flexibility in inducing a desired input rate to the system, especially in the presence of limitations as discussed above. 

In this paper we develop a mathematical model for pricing, lead time quotation and delay compensation in production and service systems with strategic customers who exhibit risk aversion. 
Specifically, we consider a make-to-order or a service system where customers arrive according to a Poisson process, and the production/service times are i.i.d. exponential random variables. This gives rise to an M/M/1 queuing model. Customers place a value on the service they receive and incur a cost per unit of time of delay. The risk aversion is modeled by a concave utility function of the net customer benefit. Based on this utility customers make individual decision to join the system or balk. Customers  cannot observe the actual state of the queue upon arrival. The individual decisions of arriving customers result in a symmetric join/balk game, for which a Nash equilibrium can be identified using game theory methodology. 
Regarding the firm's strategy, we assume that the provider announces a fixed  price of entrance, a lead time quotation and a compensation rate for the part of a customer delay which exceeds the quoted lead time. 

The main contribution of the paper is in analyzing the effect of customer risk aversion and the compensation policy on the equilibrium join/balk strategies and the resulting input rates, as well as exploring the flexibility available to the provider in inducing a range of possible input rates under various constraints on the pricing/compensation policy.  Specifically, we first characterize the symmetric customer equilibria as a function of the pricing policy. We also identify the ranges of achievable input rates when the provider must keep one or more policy parameters at a fixed value. We show that when the provider has full flexibility in the pricing/compensation policy, the effects of customer risk aversion are essentially canceled, in the sense that the optimal profit under risk-neutral customers can be approached arbitrarily closely. However in this case there is no strictly optimal policy. Finally, in numerical experiments we explore the behavior of pricing curves that reflect the provider's choices in inducing specific input rates. A key insight obtained from the analysis is that the main benefit of the lead-time and compensation option is to allow the entrance fee to remain high and the provider always prefers strategies that lead to that direction. Furthermore, as risk aversion increases, the range of achievable input rates may decrease substantially, depending on the policy constraints. In this case the provider has the highest flexibility when he or she is free to set the compensation rate, even when he or she is forced to keep the entrance fee and/or the lead-time at fixed values. 

This paper is structured as follows. In Section \ref{literature} there is a literature review. In Section \ref{modeldescription} we describe the model for customers and provider under risk aversion. In Section \ref{equilibrium} we perform equilibrium analysis. In Section \ref{loadcontrol} we present load control policies, i.e., we find the range of input rates that may be achieved by varying the remaining free policy parameters and we explore the profit maximization problem. In Section \ref{numeric} we present some numerical results, i.e., equilibrium curves, optimal profits and risk aversion sensitivity based on the negative exponential function. Conclusions and future research directions are presented in Section \ref{conclusions}. 

\section{Literature Review}\label{literature}
The implications of strategic customer behavior on the performance of a queueing system have been studied extensively in the recent years. Early works on the M/M/1 queue include \cite{naor1969regulation} and \cite{edelson1975congestion} for the observable and unobservable models, respectively. Many variations of the original models have been studied since, and a comprehensive review of the literature is provided in \cite{hassin2003queue} and \cite{hassin2016rational}. 
The focus of this paper is on the interaction between customer risk preferences and delay compensation policies, as well as their effects on the service provider profit and the load of the system.  

One of the first works to explore the effect of risk aversion in a queue with strategic customers is \cite{chen2004monopoly}, which analyzes customer equilibrium strategies, the profit maximization and the social welfare maximization problem in an unobservable setting. Customers have generally non-linear delay costs and use a concave utility function to value their net benefit from entering the system. It is shown that under non-linear utility the profit maximizing price is not socially optimal. When there is flexibility to adjust the service rate, it allows the firm to increase the entrance fee and at the same time capture the entire customer market. 
\cite{GUO2011284} consider an M/G/1 model where customers have a linear or quadratic delay cost function and  partial information on the service time distribution.   
\cite{doi:10.1080/01605682.2017.1390526} model customer risk attitude through a quadratic service utility function that involves the mean and the variance of the waiting time and allows for risk averse or risk seeking behavior. Both the unobservable and the observable settings are analyzed. It is shown that when customers are highly risk averse, then providing the queue length information hurts the service provider profits. 

Regarding the use of lead-time quotation and delay compensation as tools to mitigate the effects of risk aversion, the work most closely related to this paper is \cite{Afeche13}, who
consider an unobservable M/M/1 queue with several customer types differentiated with respect to service valuation, cost of waiting and utility function. General pricing policies are allowed that may depend on the actual sojourn time in the system. It is shown that when types are distinguishable, the service provider maximizes his/her profit by charging each type an entrance fee equal to the marginal valuation of that type and providing full compensation for the waiting cost. For indistinguishable types it is shown that the optimal linear pricing schemes that fully compensate for delay are incentive compatible. Furthermore, a simple refund policy is considered, under which a customer is fully compensated for the delay only when it exceeds a quoted lead-time. 

In an observable setting, \cite{Feng2017} also consider multiple distinguishable customer types. Customers are compensated for the delay in excess of the quoted lead time. The risk preferences are modeled by an indifference curve between entrance fee and lead-time. These boundary valuation curves are differentiated by customer type. Optimal dynamic pricing and lead-time quotation policies are derived using a Markov Decision Process model. 

Lead-time quotation policies have been extensively studied in non-strategic customer settings. 
Focusing on works that use lead-time quotations as part of a pricing strategy,
\cite{doi:10.1287/opre.1030.0089} proposes a diffusion approximation model for a Markovian queue 
with patient and impatient customers where the provider sets the service capacity and prices statically and then dynamically quotes lead-times to arriving customers and determines the service order. 
\cite{slotnick2005manufacturing} also consider dynamic lead-time quotation strategies in a make-to-order system with pricing linear in processing times. They 
focus on the impact of inaccurate production backlog information which can result in penalty fees and loss of reputation.
\cite{Plambeck2008} analyze a diffusion model for an assemble-to-order system and develop dynamic policies that include product prices, quoted leadtimes, production capacities for individual components and sequencing and expediting rules. 
\cite{doi:10.1287/opre.1080.0608} assume deterministic processing times and homogeneous customers with a nonlinear waiting cost function, and derive near optimal dynamic rules for lead-time quotation and customer sequencing. 
\cite{doi:10.1287/opre.1110.0969} develop a Semi-Markov Decision Process model for dynamic pricing and lead-time quotation in a G/M/1 system with heterogeneous service valuations and derive a threshold type structure of the optimal policy. 
\cite{doi:10.1111/j.1937-5956.2011.01248.x} consider a make-to-order system with customers classified as lead-time sensitive or price sensitive and compare a uniform quotation strategy of a single price and lead-time quotation with a differentiated strategy of offering a menu of prices and lead-time quotes. 

A related to lead-time quotation but less binding approach for addressing customer reaction to anticipated long waiting is that of delay announcements, where the provider informs incoming or waiting customers about expected completion times. The effect of delay announcements in the formation of customer equilibrium strategies has been studied in several works in the literature. We again restrict attention to the interaction of delay announcement with pricing strategies.  In an early work in this direction, \cite{Hassin1986} compares profit maximizing and socially optimal strategies in observable and unobservable systems.   
\cite{Afeche2016} develop incentive-compatible price and announced delay menus in a queue with heterogeneous risk-neutral customers. \cite{Burnetas2017} consider an unobservable queue where entering customers are informed at random times about the queue state and have the option of reneging. Pricing strategies that include separate fees for entering the system and actually receiving service are developed to optimize the social welfare. 

Finally, the problem of adjusting the input rate, the traffic and generally the load of a system has been studied extensively, although there are not many works that focus on the variability of the input rate with respect to policy parameters. This question has been explored in deterministic network flow problems under a Wardrop equilibrium framework, where it is desired to arrange the flow patterns so that no user has an incentive to change his/her route. 
\cite{Dafermos1984} and  \cite{Tobin1988} develop a sensitivity analysis methodology based on variational inequalities, and derive expressions for the derivatives of the equilibrium flows with respect to demand and cost parameters. \cite{Hai1995} incorporates queueing delays in the network links. The prevalent approach in input rate control of a queueing system is from the point of view of optimizing a profit or cost function and is related to admission control problems. We also focus here on admission control using pricing policies. In an early work, 
\cite{low1974} develops a dynamic programming model for selecting among a finite set of admission prices. \cite{ata2006} consider the joint admission and service control problem in an M/M/1 queue with adjustable arrival and service rates, under long-run average welfare maximization. They also formulate and solve an associated dynamic pricing problem. 
\cite{Paschalidis2000} and \cite{Cil2011} consider the dynamic pricing problem in a queue with more than one price sensitive customer classes.

\section{Model Description}\label{modeldescription}
We consider a Make-to-Order (MTO) system modeled as a single server Markovian queue. Potential customers arrive according to a Poisson process with rate $\Lambda$. All customers are identical and  place orders one at a time. Service times are exponentially distributed with rate $\mu$. The system operates under the first-come first-served (FCFS) discipline. When an order is placed, it brings a revenue of $R$ to the customer. There is also a waiting cost $c$ for the customer per unit of sojourn time in the system.

The provider announces an entrance fee $p$, quotes a lead-time $d$ for the service completion to each incoming customer and a compensation $l$ per unit time that the sojourn time exceeds $d$. This quotation is identical for all customers, regardless of the system state. Furthermore, incoming customers do not observe the system state upon arrival. 

Customers are strategic and decide whether to join the system or balk.  
Their decision is based on the expected net benefit that they obtain from joining. 
We further assume that customers are risk averse.
Specifically, the net benefit for a customer who joins the system and spends a total time $X$ before service completion, is equal to  $U(R-p-cX+l(X-d)^{+})$, and for a customer who balks equal to $U(0)$, where $U$ is a utility function that satisfies  the properties in the following assumption. 
\begin{assumption}\label{assump U}
\begin{itemize}
\item[i.] $U$ is strictly increasing and concave, i.e., $U'>0$ and $U'' \leq 0$.
\item[ii.] There exists $z'$ such that $U''(z) < 0$ for $z\leq z'$.
\end{itemize}
\end{assumption}
The monotonicity and concavity are standard assumptions for risk aversion. The last property ensures that $U$ is strictly non-linear for sufficiently large delay values. Assumption \ref{assump U} has the following two implications which will be used in the analysis. 
First, it is easy to show that 
\begin{equation}\label{uneginf}
lim_{z \to -\infty}U(z) = -\infty. 
\end{equation}
Second, from Jensen's inequality it follows that for any random variable $Z$ such that $P(Z<z')>0$,  
\begin{equation}\label{Jensen}
E(U(Z))<U(E(Z)).
\end{equation}
The strict inequality holds because $U$ is strictly concave in the interval $(-\infty, z')$ and the probability of $Z$ being in this interval is positive.

Following the standard approach in the unobservable framework, we restrict the analysis to symmetric mixed strategies which are determined by a common join probability $q$ for all customers. Given the provider's pricing/compensation policy determined by  $(d, p, l)$ a symmetric Nash equilibrium joining strategy  is defined as follows. Consider a tagged customer who follows strategy $h$ given that all other follow $q$. The expected utility of the tagged customer is equal to 
\begin{equation}\label{B}
B(h;q, d, p, l)=hE(U(R-p-cX+l(X-d)^{+})|q)+(1-h)U(0).
\end{equation}
The expectation in the above expression is taken with respect to the steady-state distribution of the system, under a symmetric join strategy $q$ . 
Specifically,  given a  mixed strategy $q$, the number of customers in the system evolves according to a simple M/M/1 queue process with input rate $\Lambda q$ and service rate $\mu$.
Assuming that $\Lambda q < \mu$, the steady-state distribution is geometric with parameter $\rho=\frac{\Lambda q}{\mu}$.

A strategy $q^{e}$ is Symmetric Nash equilibrium if it is a best response to itself, i.e., 
\begin{equation}\label{SNE}
B(q^{e};q^{e}) = \max_{h\in[0,1]} B(h;q^{e}).
\end{equation}

We next consider the provider. He or she takes into account the strategic behavior of the customers, by assuming that they follow a symmetric equilibrium joining strategy, as a response to the announced  pricing/compensation policy. 
We also assume that the provider is risk neutral. This is a reasonable assumption, since he or she provides the service to a large number of independent customers, thus his/her risk may be considered diversifiable and the variance of the profit over a long horizon is diminished. 
The profit function is
\begin{equation*}
 G(d, p, l) = \Lambda q^e(d,p,l) (p - l L(q^e(l,d,p), d))
\end{equation*}
which expresses the expected net profit per unit time, given  that all customers follow the equilibrium strategy $q^{e}(d, p, l)$
and $L(q,d) = E((X-d)^+|q)$
is the expected excess delay of a customer beyond the announced lead-time $d$ under symmetric mixed strategy $q$. 
 By conditioning on the number of customers in the system at the arrival instant of a random entering customer, it follows that 
\begin{eqnarray*}
L(q, d)&=&\displaystyle\sum_{n=0}^{\infty}(1- \rho) \rho^{n}\int_d^\infty (t-d) \frac{\mu^{n+1}} {n!}{t^{n}}{e}^{-\mu t}\,dt 
=\frac{{e}^{-(\mu-\Lambda q) d}} {\mu-\Lambda q}.
\end{eqnarray*}

We assume that the provider's policy space is determined by $p\leq R,  l\leq c$ and $d \geq 0$. These assumptions are consistent with pricing in a risk averse customer framework. If the provider considered entrance fees $p>R$, then he or she should have to also use $l>c$ in order to entice the customers to join. Under such a policy customers would pay a high price to join the system, betting on long delays and subsequent high compensations. This framework is beyond the scope of our analysis. 

Regarding the criteria for selecting a policy $(d,p,l)$, we broaden the scope of profit maximization and consider the provider's flexibility when the policy space is restricted by fixing one or more parameters. For each case we examine the range of achievable input rates, how this range is expanded compared to the policy that does not provide compensation, the set of policies that result in any given achievable input rate, and finally the profit maximization problem under the given restrictions in the policy space. 

In all the following analysis we will assume that under any customer equilibrium strategy the system is stable, i.e., $\Lambda q < \mu$. Even under policy  $p=R, l=c, d=0$, when it is easy to see that all customers are indifferent between joining and balking, it is not plausible that a $q\geq \Lambda/\mu$ will ever be realized in equilibrium, since no customer is expected to join the system under instability.

\section{Equilibrium Analysis}\label{equilibrium}
In this section we identify the equilibrium strategies under individual customer behavior. To do this we first analyze a tagged customer's best response to a strategy followed by all other customers. Assume that all customers join with probability $q$ and the tagged customer with probability $h$. 

From \eqref{B} the tagged customer's expected utility can be written as: $$B(h;q, d, p, l)=U(0)+hK(q, d, p, l),$$ where
$$K(q, d, p, l)= E(U(R-p-cX+l(X-d)^{+})|q)-U(0).$$ 

Under the assumption of symmetric customer strategies, there is a one-to-one correspondence between a strategy $q$ and the resulting input rate $\lambda=\Lambda q$. In the following 
we consider $K$ and $B$ as functions of $\lambda$ instead of $q$, and will use the notation $\lambda^{e}(d, p, l)=\Lambda q^e(d, p, l)$ as the equilibrium input rate under $(d, p, l)$. 

Before we proceed to equilibrium analysis, we explore some properties of $K(\lambda, d,p,l)$.
It is first easy to see that $K$ is strictly decreasing in $p$ and $d$ and strictly increasing in $l$. Regarding monotonicity in $\lambda$, $K$ can be written as follows:
\begin{eqnarray*}
K(\lambda, d, p, l)
&=&E(U(R-p-(c-l)X+ l((X-d)^{+}-X))|\lambda)-U(0)\\
&=&E(U(R-p-(c-l)X- l \min(X,d))|\lambda)-U(0).
\end{eqnarray*}
In steady state the sojourn time of a customer is exponentially distributed with rate $\mu-\lambda$. Therefore, $X$ is stochastically increasing in $\lambda$, and since $U$ is decreasing in $X$, it follows that $K$ is 
decreasing in $\lambda$. In particular, when $c=l$ and $d=0$, $K$ is constant and equal to $U(R-p)-U(0)$, whereas if $l<c$ or $d>0$, $K$ is strictly decreasing in $\lambda$.

We next show some useful limiting properties.
First, when $l=0$ or $d \to \infty$, there is effectively no delay compensation; therefore, the problem reduces to \cite{chen2004monopoly} under a linear waiting cost function, i.e., 
\begin{equation}\label{KCF}
  K_{CF}(\lambda,p,c) = E(U(R-p-cX)|\lambda)-U(0).
\end{equation}
Let $\lambda^{e}_{CF}(p;c)$ be the corresponding equilibrium input rate.
Similarly, when $d=0$, customers are compensated from the first minute of delay, which reduces the problem to one with no delay compensation and waiting cost equal to $c-l$; therefore, $K(\lambda,0,p,l)=K_{CF}(\lambda,p,c-l)$. 

The behavior of $K$ as the input rate approaches $\mu$ is more involved. 
When $l<c$, the quantity  
$R-p-(c-l)X- l \min(X,d)$ can take arbitrarily small values with positive probability, thus, from \eqref{Jensen}, 
$$
K(\lambda,d, p,l) < U(E(R-p-(c-l)X- l \min(X,d)|\lambda)) - U(0).
$$
However, $E(R-p-(c-l)X- l \min(X,d)|\lambda) < R-p -\frac{c-l}{\mu-\lambda}$, thus, from \eqref{uneginf}, 
\begin{equation}\label{Kneginf}
\lim_{\lambda \to \mu-} K(\lambda, d, p,l) = -\infty.
\end{equation}
This limiting behavior of $K$ is expected, since a customer who joins a very congested system is faced with unbounded delay and if he/she is partially compensated, the net benefit may take arbitrarily low values. On the other hand, for $l=c$, 
$$
K(\lambda, d, p, l) = U(E(R-p- c \min(X,d)|\lambda)) - U(0),
$$
and since $\lim_{\lambda \to \mu-} P(X>d) = 1$ for all $d>0$, it follows that 
\begin{equation}\label{Kneginfc}
\lim_{\lambda \to \mu-} K(\lambda, d, p,l) = U(R-p-cd) - U(0).
\end{equation}
In this case, although the delay of a joining customer is unbounded, he/she assumes the cost only for the first $d$ units of the delay and is fully compensated for the remaining part, therefore the expected benefit is finite and may even be positive, depending on the values of $d, p$ and $l$. 
Note that as $d\to \infty$, the compensation is effectively canceled, thus the right-hand side of \eqref{Kneginfc} also tends to $-\infty$, which is consistent with the limit of \eqref{KCF} as $\lambda \to \mu-$.

The tagged customer maximizes his/her expected utility $B(h;\lambda, d,p,l)$ in $h$ given $\lambda$. Therefore, if $K(\lambda, d,p,l)<0$ his/her unique best response to input rate $\lambda$ is equal to 0, if $K(\lambda, d,p,l)>0$ the unique best response is equal to 1, and if $K(\lambda, d,p,l)=0$ any strategy in $[0,1]$ is a best response. 
Since $K(\lambda, d, p, l)$ is decreasing in $\lambda$, the best response $h^{*}(\lambda)$ is also decreasing, which implies an Avoid the Crowd (ATC) behavior for customers (c.f. \cite{hassin2003queue}, p. 6-7). This means that as traffic increases, customers are more reluctant to join, although they are partially compensated for their delay.

In the next theorem we characterize the existence and uniqueness of symmetric equilibrium rates under various cases for the values of $p, l, d$.
\begin{theorem} \label{theopld}
The symmetric equilibrium input rates are characterized as follows. 
\begin{itemize}
\item[i.] For $p=R$ there are the following cases. 

(a) If $l=c$ and $d=0$, then any $\lambda^e \in [0,\mu) \cap [0,\Lambda]$ is a symmetric equilibrium.

(b) In all other cases, $\lambda^e=0$ is the unique symmetric equilibrium.

\item[ii.] For $p<R$ there are the following cases. 

(a) If $K(0, d,p,l) \leq 0$, then  $\lambda^e=0$ is the unique symmetric equilibrium.

(b) If $K(0, d,p,l) > 0$ and $\lim_{\lambda \to \mu-} K(\lambda, d, p,l) <0$, then the unique symmetric equilibrium is equal to $\min(\lambda^0,\Lambda)$, where $\lambda^0$ is the solution of $K(\lambda, d, p, l)= 0$. 

(c) If $\lim_{\lambda \to \mu-} K(\lambda, d, p,l) \geq 0$, for $\Lambda \geq \mu$ a symmetric equilibrium does not exist, while for $\Lambda <\mu$, $\lambda^e=\Lambda$ is the unique symmetric equilibrium.
\end{itemize}
\end{theorem} 

\begin{proof}

i) In case (a), $K(\lambda, 0, R, c)= 0$ for all $\lambda<\mu$, therefore any $\lambda^e \in[0,\mu) \cap [0,\Lambda]$ is a symmetric equilibrium.

In case (b), if $l=c$ and $d>0$ then $K(\lambda, d, R, c)=E(U(-c \min(X,d)|\lambda)-U(0)<0$ for all $d>0$ and $\lambda<\mu$. Therefore $\lambda^e=0$ is a unique symmetric equilibrium. On the other hand, if $l<c$ and $d\geq 0$, then $K(\lambda, d, R, l)= E(U( -cX+l(X-d)^{+}|\lambda))-U(0)<0$ for all $\lambda<\mu$, because $(X-d)^{+} \leq X$ and $l<c$.  Therefore $\lambda^e=0$ is a unique symmetric equilibrium.

ii) The proof follows immediately from the form of the best response function and the monotonicity of $K(\lambda, d,p,l)$ in $\lambda$. 
In particular, in case (b) $K$ is not constant, thus it is strictly decreasing in $\lambda$ and the equation $K(\lambda, d,p,l)=0$ has a unique solution $\lambda^0$.
\end{proof}

From Theorem \ref{theopld} it follows that when the entrance fee is equal to the service reward, 
then no customer will join if the delay compensation is not in full and from the first minute, i.e., $l=c$ and $d=0$. If this is true, then all customers are indifferent and any input rate that does not lead to instability is equilibrium. 
When $p<R$, the equilibrium is unique and varies from $0$ to $\Lambda$, depending on the policy parameters, except for case ii (c), where for $\Lambda \geq \mu$ a symmetric equilibrium does not exist. In this case all customers have a positive benefit from joining as long as $\lambda < \mu$, however for $\lambda \geq \mu$ the system becomes unstable and any joining customer has zero probability of receiving the service reward in a finite time. 

The next proposition establishes monotonicity properties of $\lambda^e$ with respect to the policy parameters. 

\begin{proposition}\label{lambdae properties}
When a unique symmetric equilibrium input rate $\lambda^e(d,p,l)$ exists, it has the following properties:
\begin{align*}
\frac{\partial \lambda^e}{\partial p}\leq 0,\quad
\frac{\partial \lambda^e}{\partial d}\leq 0,\quad
\frac{\partial \lambda^e}{\partial l}\geq 0.
\end{align*}

\end{proposition}
\begin{proof}
We only prove the monotonicity in $p$ and the other two are similar. 
When a unique $\lambda^e$ exists and $K(\lambda^e)<0$ (in which case $\lambda^e=0$) or 
$K(\lambda^e)>0$ (in which case $\lambda^e=\Lambda$), then $\frac{\partial \lambda^e}{\partial p}= 0$, since the inequality is preserved for a small change in $p$.
When $K(\lambda^e)=0$, then 
$$\frac{\partial K}{\partial p}+\frac{\partial K}{\partial \lambda} \frac{\partial \lambda^e}{\partial p}=0.$$ 
However in this case, from Theorem \ref{theopld} ii(b) $\frac{\partial K}{\partial \lambda} <0$. In addition, $\frac{\partial K}{\partial p} <0$, thus, $\frac{\partial \lambda^e}{\partial p} <0$. 
\end{proof}

\section{Load Control, Policy Flexibility and Profit Maximization}\label{loadcontrol}

In this section we consider the problem from the point of view of the provider.
In particular, we are interested in the degree of flexibility that pricing/delay compensation policies offer to the provider to control the input rate of customers as well as maximize his/her profits under various constraints in the policy parameters. 

There are several occasions where a provider may be forced to fix some policy parameters for exogenous reasons. For example, the entrance fee $p$ may be fixed due to legislation or company policy. In this case the compensation is useful as an extension of the pricing strategy. In other situations, the firm may offer compensation for customers who wait more than an advertised time period. In such cases, the provider may have flexibility in adjusting the entrance fee and the compensation rate, but must keep the lead-time quotation at a fixed level. There are also cases where two policy parameters must be set, for example compensation rate and lead-time quotation, etc. 

In situations such as the above, several questions arise. Assume that the pricing and compensation policy is restricted in some parameters. A common approach is to identify the policy that maximizes the provider's profit under the given constraints, with no regard to the resulting input rate.
Herein we adopt a more general viewpoint. We find the range of input rates that may be achieved by varying the remaining free policy parameters, assess the degree of flexibility in achieving any particular input rate  and identify the values of the free parameters that maximize the profit for the given input rate. This approach effectively views the input rate as an additional parameter controllable by the provider. It may be very useful in situations where in addition to profits the provider must also consider additional factors. For example, startup companies that introduce new products or services may need a pricing strategy that generates a given demand pattern, not necessarily optimal during the introduction phase. This is usually done for achieving the most appropriate market penetration rate in the long term, depending on recognition of the company or the brand name, the intensity of competition, the available resources to ensure the desirable service levels, etc.

In the rest of the paper we refer to the case where the input rate $\lambda$ is explicitly considered part of the provider's strategy as the load control problem. 

Before we proceed, we formalize the notion of achievability of an input rate, either unconstrained or with some of the policy parameters fixed.

\begin{mydef}\label{def_achievable}
\normalfont
\begin{enumerate}
\item
  An input rate $\lambda$ is achievable if there exists a pricing/compensation policy $(d,p,l)$ under which $\lambda$ is the unique equilibrium, i.e.,  $\lambda^e(d,p,l) = \lambda$.
\item
  For an entrance fee $p$,  an input rate $\lambda$ is $p$-achievable if there exists a pair $(d,l)$, such that $\lambda^e(d,p,l)= \lambda$.
\item
  For a pair $(p,l)$ of entrance fee and cost compensation, $\lambda$ is $(p,l)$-achievable if there exists a lead-time $d$ such that $\lambda^e(d,p,l)=\lambda$.
\end{enumerate}
\end{mydef}

The definition of  $l, d, (d,p)$ and $(d,l)$-achievable input rates is similar. 

In these definitions it is required that the corresponding policies in each case result in $\lambda$ as the  unique equilibrium input rate. This is so, because if multiple equilibria exist under a policy $(d,p,l)$, then there is no way to ensure that a particular of those will occur, without using additional controls.

In the following subsections we analyze the load control problem and identify the ranges of achievable input rates, with the provider's policy flexibility varying from low level, where two of the policy parameters are fixed, to full flexibility where the provider is free to set all policy parameters as desired. Note that Theorem \ref{theopld} about equilibrium analysis can be considered as the other extreme case of no flexibility since, if all parameters are fixed, there is either a single or no achievable rate. 

\subsection{Low Policy Flexibility}\label{twofixed}

From the equilibrium analysis in Theorem \ref{theopld} and the monotonicity properties of function $K(\lambda,d,p,l)$  it follows that if two policy parameters, e.g., $(p,l)$, are fixed and $\lambda$ is $(p,l)$-achievable with $0< \lambda < \Lambda$, then the third policy parameter $d$ is uniquely determined and  it is  such that $K(\lambda, d, p, l) =0$.
On the other hand, if $\Lambda$ is $(p,l)$-achievable, then it is achieved by any lead-time $d$ such that $K(\Lambda, d, p, l) \geq 0$. 
From the monotonicity and continuity of $K$  in $d$, this inequality holds for $d$ at or below a maximum value.
In both cases we define the required lead-time $d^e(\lambda,p,l) = \max\{d : \lambda^e(d,p,l) = \lambda\}$, as the value that maximizes the provider's profit among all that achieve $\lambda$ under fixed $(p,l)$. 

Similarly we define the required entrance fee
$p^e(\lambda,d,l) =  \max\{p : \lambda^e(d,p,l) = \lambda\}$ 
and the required compensation cost
$l^e(\lambda,d,p) = \min\{l : \lambda^e(d,p,l) = \lambda\}$.

In the next theorem we establish the intervals of achievable input rates under pricing/compensation policies with two parameters fixed. Recall that the case $d \to  \infty$ is equivalent to setting $l=0$, because in this case the customer will never receive the compensation.

\begin{theorem}\label{achievable lambda}
\begin{itemize}
\item[i.] For a pair $(p,l)$ of entrance fee and delay compensation the achievable range of $\lambda$ is as follows.

If $p=R$, then the only achievable input rate is $\lambda = 0$.

If $p<R$ and $l<c$, then an input rate $\lambda$ is $(p,l)$-achievable if and only if $\lambda^{e}_{CF}(p;c) \leq \lambda \leq\lambda^{e}_{CF}(p;c-l)$.

If $p<R$ and $l=c$, then an input rate $\lambda$ is $(p,l)$-achievable if and only if
$\lambda \in [\lambda^{e}_{CF}(p;c) ,\mu) \cap [\lambda^{e}_{CF}(p;c),\Lambda]$.

\item[ii.] For a pair $(d,l)$ of lead-time and delay compensation with $d < \infty$,  the achievable range of $\lambda$ is as follows.

  If $l=c, d=0$, then the only achievable input rate is $\lambda = \Lambda$ if $\Lambda < \mu$, otherwise there is no achievable input rate.

  If $l<c$ or $d>0$, then an input rate $\lambda$ is  $(d,l)$-achievable if and only if 
  $0\leq\lambda\leq\lambda_{\max}(d,l)$, where
  \[\lambda_{\max}(d,l) =
    \begin{cases}
      0, & \quad \mbox{if \ } K(0, d, 0, l) \leq 0 \\  
      \lambda_0(d,l) & \quad  \mbox{if \ } K(\Lambda, d, 0, l) <  0 < K(0,d,0,l) \\  
      \Lambda, & \quad  \mbox{if \ } K(\Lambda, d, 0, l) \geq  0 
    \end{cases}
  \]
  and $\lambda_0(d,l)$ is the unique solution of $K(\lambda,d,0,l)=0$ in $\lambda$.
  
\item[iii.] For a pair $(d,p)$ of lead-time and entrance fee with $d<\infty$, the achievable range of $\lambda$ is as follows.

  If $p=R$, then the only achievable input rate is $\lambda = 0$.

  If $p<R$, then an input rate $\lambda$ is  $(d,p)$-achievable if and only if
   $\lambda^{e}_{CF}(p;c)\leq\lambda\leq\lambda_{\max}(d,p)$, where
  \[\lambda_{\max}(d,p) =
    \begin{cases}
      0, & \quad \mbox{if \ } K(0, d, p, c) \leq 0 \\  
      \lambda_0(d,l) & \quad  \mbox{if \ } K(\Lambda, d, p, c) <  0 < K(0,d,p,c) \\  
      \Lambda, & \quad  \mbox{if \ } K(\Lambda,d,p,c) \geq  0 
    \end{cases}
  \]
  and $\lambda_0(d,p)$ is the unique solution of $K(\lambda,d,p,c)=0$ in $\lambda$.

\end{itemize}
\end{theorem}
\begin{proof}

 i)
 If $p=R$, then from case(i) of Theorem \ref{theopld}, it follows that the only arrival rate that can be unique equilibrium is $\lambda=0$.  
 
Next consider the case $p<R, l<c$.
Then $\lambda=0$ is achievable if there exists $d$ (including the limit $d\to \infty$ which corresponds to no compensation), such that $K(0,d,p,l)\leq 0$. Since $K(0,d,p,l)$ is decreasing in $d$, this is true if and only if $\lim_{d \to \infty} K(0, d,p,l) \leq 0$, i.e., $K_{CF}(\lambda,p,c) \leq 0$ thus, $\lambda^{e}_{CF}(p;c)=0$. Similarly, $\lambda=\Lambda$ is achievable if there exists $d$ such that $K(\Lambda,d,p,l)\geq 0$. This is true if and only if $K(\Lambda,0,p,l)\geq 0$, i.e., $K_{CF}(\Lambda,p,c-l)\geq 0$ and $\lambda^{e}_{CF}(p;c-l)=\Lambda$.

Finally, a rate $\lambda$ such that $0<\lambda<\Lambda$ is achievable if and only if there exists $d$ such that $K(\lambda,d,p,l)=0$.
From the monotonicity of $K$ in $d$, this is true if and only if $K_{CF}(\lambda,p,c) \leq 0 \leq K_{CF}(\lambda,p,c-l)$, i.e., $\lambda^{e}_{CF}(p;c) \leq \lambda \leq\lambda^{e}_{CF}(p;c-l)$.

The proof for the case $p<R, l=c$ is entirely similar, with the only difference that for $l=c$ the equilibrium rate $\lambda^e_{CF}(p;c-l)=\lambda^e_{CF}(p;c)$ does not exist. However in this case by letting $d\to 0$, all rates below $\mu$ and $\Lambda$ can be achieved. 

ii)
First assume $l=c, d=0$. From Theorem \ref{theopld} only $\Lambda$ is achievable, if $\Lambda <\mu$, by setting any $p<R$. 

Next consider $l<c$ or $d>0$. From Theorem \ref{theopld}(ib), rate $\lambda=0$ is always $(d,l)$-achievable, by setting $p=R$. From Theorem \ref{theopld}(ii), rate $\lambda=\Lambda$ is $(d,l)$-achievable if there exists $p<R$ such that $K(\Lambda, d,p,l)\geq 0$. Since $K(\Lambda, d,p,l)$ is decreasing in $p$, this is true if and only if $K(\Lambda, d, 0,l) \geq 0$.
Finally, a rate $\lambda$ in the interval $0 < \lambda < \Lambda$ is a unique equilibrium if and only if there exists $p$ such that $K(\lambda, d, p, l)=0$. Since $K(\lambda, d, p, l)$ is decreasing in $p$ and $K(\lambda, d, R, l) < 0$,
$\lambda$ is $(d,l)$-achievable if and only if $K(\lambda, d, 0, l) \geq 0$.

We have thus found that any $\lambda > 0$ is $(d,l)$-achievable if and only if $K(\lambda, d, 0, l) \geq 0$. Since $K(\lambda, d, 0,l)$ is strictly decreasing in $\lambda$, it follows that the inequality is true for $\lambda \leq \lambda_{\max}(d,l)$ as defined above. 

iii)
If $p=R$, then rom Theorem \ref{theopld}(ia) the only achievable input rate is $\lambda = 0$.

Next consider $p<R$. Then, $\lambda=0$ is achievable if there exists $l$ such that $K(0,d,p,l)\leq 0$, and since $K(0,d,p,l)$ is increasing in $l$, this is true if and only if $K(0,d,p,0) \leq 0$.
Similarly,  $\lambda=\Lambda$ is achievable if there exists $l$ such that $K(\Lambda,d,p,l)\geq 0$, i.e.,  if and only if $K(\Lambda,d,p,c) \geq 0$.
Finally, a rate $\lambda$ such that $0 < \lambda < \Lambda$ is achievable if there exists $l$ such that $K(\lambda, d,p,l)=0$, which holds if  $K(\lambda, d, p,0)\leq 0 \leq K(\lambda, d,p,c)$. 

We have thus found that any $\lambda \in [0,\Lambda]$ is $(d,l)$-achievable if and only if  $K(\lambda, d, p,0)\leq 0 \leq K(\lambda, d,p,c)$. 
Since $K(\lambda, d,p,0) = K_{CF}(\lambda, p, c)$, the first inequality is true if and only if $\lambda \geq \lambda_{CF}(p;c)$.
The second inequality,  since $K(\lambda, d, p,c)$ is decreasing in $\lambda$, holds
for $\lambda  \leq \lambda_{max}(d,p)$, where $\lambda_{max}(d,p)$ is as defined above. 

Summarizing the above, we obtain the achievable range $ \lambda_{CF}(p;c) \leq \lambda  \leq \lambda_{max}(d,p)$.
\end{proof}

\subsection{High Policy Flexibility} \label{onefixed}

In this subsection we explore to what extend the provider can control the input rate when only one policy parameter is fixed and he or she has flexibility on the remaining two parameters. For any value of the fixed policy parameter, there are generally infinite pairs of values of the other two parameters which result in a desired achievable input rate. They can be computed from the equilibrium conditions and have the form of pricing curve. For example, when the entrance fee $p$ is fixed, the set of $(d, l)$ pairs that result in a given input rate $\lambda$ form the corresponding $(d,l)$-pricing curve.
The provider may want to maximize his/her profit along any such pricing curve. 

In the next theorem we establish the intervals of achievable input rates. 
In the section on numerical experiments we derive the form of the pricing curves and consider the corresponding profit maximization problem, for the case of a negative exponential utility function. 

\begin{theorem}\label{onefixedachievable}
\begin{itemize}
\item[i.] For a fixed entrance fee $p$ an input rate $\lambda$ is $p$-achievable if and only if:   
$\lambda\in[\lambda_{CF}(p;c),\mu) \cap [\lambda_{CF}(p'c),\Lambda]$. 
\item[ii.] For a fixed delay compensation $l<c$ an input rate $\lambda$ is $l$-achievable if and only if:    
$\lambda\in[0,\min(\lambda^e(0,0,l),\Lambda)] \cap [0,\mu)$. For $l=c$ the achievable input rate interval is $[0,\Lambda] \cap [0,\mu)$.
\item[iii.] For a fixed lead-time quotation $d>0$ an input rate $\lambda$ is $d$-achievable if and only if:   
$\lambda\in[0,\min(\lambda^e(d,0,c),\Lambda)] \cap [0,\mu)$. For $d=0$ the achievable input rate interval is $[0,\Lambda] \cap [0,\mu)$.
\end{itemize}
\end{theorem}

\begin{proof}
The proof is based on identifying the worst and best parameter combination for customers in each case. Specifically, for a fixed entrance fee $p$, the worst case for customers is to receive no compensation, which corresponds to setting $l=0$ or letting $d \to \infty$, and the best case is to be fully compensated from the entire time they spend in the system. For a fixed compensation rate $l$, the worst case for customers is to pay an entrance fee equal to their value of service for a lead-time that tends to infinity. On the other hand the best case is to receive the service with free entrance and be compensated for the entire time they spend in the system. Finally, for a fixed lead-time $d$, the worst scenario for customers is to pay an entrance fee equal to the service reward and receive a compensation rate arbitrarily close to zero, and the best scenario is to receive the service with free entrance and be compensated at full rate $l=c$. 
\end{proof}

From Theorem \ref{onefixedachievable} we can also assess the flexibility available to the service provider by the lead time quotation and compensation policy, compared to single entrance fee pricing. Specifically, from part (i) of the Theorem it follows that if the entrance fee $p$ is fixed and no delay compensation is allowed, then the only achievable input rate is $\lambda_{CF}(p;c)$. On the other hand, when the provider has the additional flexibility of offering a lead time quotation and delay compensation, then by setting parameters $d$ and $l$ appropriately, he or she can achieve all higher values of $\lambda$, either up to the market size $\Lambda$ or the stability limit $\mu-$. Therefore, the enhanced policy allows to increase the input rate to any desirable level above the base value achievable with a single entrance fee. In addition, any rate in this interval can be achieved by an infinite collection of parameter values lying on the corresponding $(d,l)$-pricing curve.

In Section \ref{numeric} we explore the behavior of the pricing curves and the corresponding profit maximization problem under the high flexibility setup.

\subsection{Full Policy Flexibility and Profit Maximization}\label{profitmax}

In this section we focus on the case when the provider has flexibility in all three policy parameters $d,p,l$. Combining the results of the previous two cases it follows that the provider can induce any value of $\lambda$, i.e., the region of achievable input rates is $[0, \mu) \cap [0,\Lambda]$. 

In this subsection we also analyze the provider's profit maximization problem under full flexibility. The main result is that the possibility of delay compensation allows the provider to regain essentially the entire loss in profits due to customer risk aversion. However due to the properties of the customer equilibrium strategies there is no strictly optimal policy; instead, there are policies that achieve profit arbitrarily close to the optimal level under risk-neutral customers. 

The provider's profit function can be expressed in several equivalent forms. To do this we start with a general expression. Let: 
\begin{equation} \label{g1}
G_1(\lambda, d, p, l)=\lambda \left( p-l\frac{e^{-(\mu-\lambda)d}}{\mu-\lambda} \right), 
\end{equation} 
denote the provider's profit under policy $(d, p, l)$ and input rate $\lambda$, not necessarily in equilibrium.
Assuming that the customers respond strategically to policy $(d, p, l)$ the provider's actual profit function is equal to: 
\begin{equation} \label{profit}
G(d, p, l)=G_1(\lambda^{e}(d, p, l), d, p, l)=\lambda^{e}(d, p, l)\left( p-lL^{e}(d, p, l)\right), 
\end{equation}
$L^{e}(d, p, l)=E\left( (X-d)^{+}|\lambda^{e}(d, p, l)\right)$ is the expected lateness in steady state under the equilibrium input rate $\lambda^{e}(d, p, l)$. 

As discussed in the beginning of this section, if the equilibrium input rate is not unique under a policy $(d,p,l)$, then the provider cannot enforce any particular value.  Therefore, in order for the profit maximization problem to be well defined, we restrict the pricing policies to those that  result in achievable input rates.

Using the notion of achievable input rates, we can also express profit maximization as a two stage optimization problem where in the second stage the provider determines the pricing/compensation policy that achieves any achievable input rate with the maximum possible profit and in the first stage determines the optimal achievable input rate, i.e., 

\begin{equation}\label{Gstar}
G^{*}=\sup\limits_{d, p, l}G(d, p, l) =\sup\limits_{\lambda}H(\lambda)\\
\end{equation}
where,
\begin{equation*}
H(\lambda)=\sup\limits_{d, p, l}\lbrace G(d, p, l) :\lambda^e(d, p, l)=\lambda\rbrace .
\end{equation*}

An input rate $\lambda$ is achievable if and only if it is $(p,l)$-achievable for at least one pair $(p,l)$. Therefore, $H(\lambda)$ can be expressed as
\begin{equation}\label{Hdef}
H(\lambda)=\sup\limits_{p, l} G_1(\lambda, d^e(\lambda, p, l), p, l)  ,
\end{equation}
since $d^e(\lambda, p, l)$ is the lead-time that maximizes profits among those that achieve $\lambda$ under $(p,l)$. 

Note that we could equivalently take the supremum with respect to $(d,p)$ or $(d,l)$ and use the required compensation $l^e$ or entrance fee $p^e$, respectively.

We will show below that in general the supremum in the above expressions cannot be attained. The reason is that for any achievable $\lambda$, by using a pricing/compensation policy $(d,p,l)$ that approaches $(0,R,c)$ the profit inreases. However, the policy $(0,R,c)$ itself does not correspond to any achievable $\lambda$, since under it all $\lambda$ are equilibrium and the provider cannot enforce any of them without further actions. We formalize these arguments in Propositions \ref{GRNGU} and \ref{Hsup} and Theorem \ref{eoptimal}.

\begin{proposition}\label{GRNGU}
   For any utility function $U$ that satisfies Assumption \ref{assump U}, the provider's profit is bounded above by 
$$G_1(\lambda, d^{e}(\lambda, p, l), p, l)< \lambda\left( R-\frac{c}{\mu-\lambda}\right),$$ 
for all $\lambda>0$ and  $(p, l)$ such that $\lambda$ is $(p,l)$-achievable.
\end{proposition}

\begin{proof}
 If $\lambda < \Lambda$, then the equilibrium condition is 
  $K(\lambda, d, p, l)  =0,
  $
  while for $\lambda = \Lambda$,
  $
  K(\Lambda, d, p, l)\geq 0.
  $
  In both cases
  $$
  U(0) \leq E(U(R-p-cX+l(X-d^{e}(\lambda, p, l))^{+})).
  $$
  
  However, since $\lambda>0$ and $(p,l)$-achievable, it follows from Theorem  \ref{theopld} that $l<c$ or $d^{e}(\lambda, p, l))>0$. In addition, $X$ is exponentially distributed, thus it can take arbitrarily large values with positive probability. Therefore, from Assumption \ref{assump U} :
  $$
  P(R-p-cX+l(X-d^{e}(\lambda, p, l))^{+}<z') > 0,
  $$
  and from \eqref{Jensen} it follows that
  $$
  E(U(R-p-cX+l(X-d^{e}(\lambda, p, l))^{+})) < U(E(R-p-cX+l(X-d^{e}(\lambda, p, l))^{+})). 
  $$

Combining the last two inequalities we obtain 
$$U(0) < U(R-p-\frac{c}{\mu-\lambda}+lL(\lambda, d^{e}(\lambda, p, l)))
$$
and, since $U$ is increasing:
$$0< R-p-\frac{c}{\mu-\lambda}+lL(\lambda, d^{e}(\lambda, p, l)).$$
from which part (i) follows.
\end{proof}

We next show that $H(\lambda)=\lambda\left( R-\frac{c}{\mu-\lambda}\right)$ but the supremum in \eqref{Hdef} is not attained.

\begin{proposition}\label{Hsup}
For any $\lambda \in (0,\mu) \cap (0,\Lambda]$, 
\begin{itemize}
\item[i.] $H(\lambda)=\lambda\left( R-\frac{c}{\mu-\lambda}\right)$.
\item[ii.] There is no $(p, l)$ such that $\lambda$ is $(p, l)$-achievable and $G_1(\lambda, d^{e}(\lambda, p, l), p, l)=H(\lambda).$
\end{itemize}
\end{proposition}
\begin{proof}
  i) Fix a  $\lambda \in (0,\mu) \cap (0,\Lambda]$.
  From Proposition $\ref{GRNGU}$ it follows that
  $$
  H(\lambda) =  \sup_{p, l}G_1(\lambda, d^{e}(\lambda, p, l), p, l) \leq
  \lambda\left( R-\frac{c}{\mu-\lambda}\right).
  $$

 For the above to  hold as equality,  it must also be true  that for any $\epsilon>0$ there exist $p,l$ such that $\lambda$ is $(p,l)$-achievable and
 \begin{equation}
   \label{eq:Gconv}
 G_1(\lambda, d^{e}(\lambda, p, l), p, l)>\lambda\left( R-\frac{c}{\mu-\lambda}\right)-\epsilon.   
 \end{equation}

To prove this, we will show that a pair $(p,l)$ of the form $p=R-\delta, l=c$, for $\delta>0$ sufficiently small can achieve both requirements.

We first show that there exists $\tilde{\delta} > 0$, such that  $\lambda$ is $(R-\delta,c)$-achievable for $\delta < \tilde{\delta_1}$. 
From Theorem  $\ref{achievable lambda}$ we know that for $p<R$ and $l=c$, $\lambda$ is achievable if and only if $\lambda^e_{CF}(p;c) \leq \lambda$. Furthermore, $\lambda^e_{CF}(R;c)=0$, since no customer will join if the entrance fee is equal to $R$ and there is no waiting cost compensation.
Thus, there exists a $\tilde{\delta}>0$ such that for any $\delta<\tilde{\delta}: \lambda^{e}_{CF}(R-\delta;c) \leq \lambda$.

We next prove (\ref{eq:Gconv}). 
For any $\delta<\tilde{\delta}$, the required lead-time $d^{e}(\lambda, R-\delta, c)$ is uniquely defined. In addition, the provider's profit can be expressed as 
\begin{eqnarray*}
G_1(\lambda, d^{e}(\lambda, R-\delta, c), R-\delta, c)&=&\lambda \left( R-\delta-c \frac{e^{-(\mu-\lambda)d^{e}(\lambda, R-\delta, c)}}{\mu-\lambda} \right)\\
&=&\lambda\left( R-\frac{c}{\mu-\lambda}\right)  - h(\lambda,\delta),
\end{eqnarray*} 
where
$$
h(\lambda, \delta) = 
\lambda \left(\delta - c\frac{1-e^{-(\mu-\lambda)d^{e}(\lambda, R-\delta, c)}}{\mu-\lambda} \right).
$$
Therefore, to prove  \eqref{eq:Gconv}, it suffices to show that
$\lim_{\delta \to 0}h(\lambda, \delta)=0$,
or equivalently that
$\lim_{\delta \to 0} d^{e}(\lambda, R-\delta, c)=0$.
To show the last property, note that since $d^{e}(\lambda, R-\delta, c)$ is increasing in $\delta$, this limit exists, and is generally equal to some $d'\geq 0$. Assume that $d'>0$. Then, $d^{e}(\lambda, R-\delta, c) \geq d'$ for all $\delta >0$.
On the other hand, by definition of $d^e$ it is true that 
$\lambda^e(d^{e}(\lambda, R-\delta, c), R-\delta, c) = \lambda$ for all $\delta <\tilde{\delta}$. Since $\lambda^e$ is decreasing in $d$, it follows that $\lambda^e(d', R-\delta, c) \geq \lambda$ for all $\delta <\tilde{\delta}$, therefore,
$\lim_{\delta \to 0} \lambda^{e}(d', R-\delta, c) \geq \lambda$, which is a contradiction, because from Theorem \ref{theopld} it follows that
$\lim_{\delta \to 0} \lambda^{e}(d', R-\delta, c) =0$.

ii) From Proposition \ref{GRNGU} it follows that $G_1(\lambda, d^{e}(\lambda, p, l), p, l)<\lambda \left ( R-\frac{c}{\mu-\lambda} \right )$ for any $p<R, l<c$. Therefore, the supremum cannot be attained for any $p, l$ such that $\lambda$ is $(p, l)$-achievable. 
\end{proof} 

The quantity $\lambda\left( R-\frac{c}{\mu-\lambda}\right)$ is equal to the social benefit per unit time under risk neutral customers. It is well known \citep{edelson1975congestion} that under risk neutrality the provider's profit maximization problem coincides with that of social benefit maximization, i.e., the service provider sets a price that induces the socially optimal input rate and, by doing that, reaps the entire social benefit. A significant consequence of Propositions  \ref{GRNGU} and \ref{Hsup} is that even when customers are risk averse the provider can induce any input rate he prefers with a profit arbitrarily close to the social benefit under that rate. Therefore, by exploiting the lead-time/delay compensation option, he or she can regain essentially all the lost profit and thus revert the effects of customer risk aversion. 

Propositions \ref{GRNGU} and \ref{Hsup} solved the second stage optimization problem in \eqref{Hdef}. Returning to the first stage problem in \eqref{Gstar}, we observe that $H(\lambda)$ is maximized for
$\lambda= \lambda^{*}=\min(\mu-\sqrt{\frac{c \mu}{R}},\Lambda)$, as is also well known from the standard unobservable problem. Thus, the supremum in \eqref{Gstar} is attained by $\lambda^*$, and $G^{*}=H(\lambda^{*})$.

Although $\lambda^{*}$ is the optimal input rate that the provider must induce to  maximize $H(\lambda)$, as we have seen it is not possible to attain the profit $H(\lambda^*)$ by any $(d,p,l)$ policy. However it is possible to construct $\epsilon$-optimal policies, which approach $H(\lambda^*)$ arbitrarily close, and from the proof of Proposition \ref{Hsup} it follows that one way to do that is to offer full delay compensation and set an entrance fee lowe but sufficiently close to the service reward $R$. For any $p$ set, the required lead-time that will induce $\lambda^*$ is equal to $d^e(\lambda^*, p, c)$.  

All the above is summarized in Theorem \eqref{eoptimal} below.

\begin{theorem}\label{eoptimal}
\begin{itemize}
\item[i]  $G^{*}=\sup\limits_{d, p, l} G(d, p, l)=\max\limits_{\lambda\in(0,\mu)}H(\lambda)=H(\lambda^{*})$, where $\lambda^{*}=\min(\mu-\sqrt{\frac{c \mu}{R}}, \Lambda)$.
\item[ii] There is no optimal solution for the provider, i.e., the $\sup$ is not attained under any policy $(d, p, l)$. However there exist $\epsilon$-optimal policies, i.e., for any $\epsilon>0$ there exists $\delta>0$ such that: $$G\left( d^e(\lambda^{*}, R-\delta, c), R-\delta, c) \right)>G^{*}-\epsilon.$$
\end{itemize}
\end{theorem}

\section{Numerical Results}\label{numeric}
In this section we perform computational experiments to obtain further insights on the behavior of the pricing curves and the provider's optimal profit under one fixed policy parameter, as well as on the effect of customer risk aversion. For the computations we use  
$$
U(X)=\frac{1-e^{-rX}}{r},
$$
with $r>0$, which is a utility function in the class of Constant Absolute Risk Aversion  (CARA), with absolute risk aversion equal to $r$. Note that $U$ is decreasing in $r$ for all $x$,  and 
as $r$ diminishes to zero the utility function converges to the risk neutral form $U(X)= X$. 

For this utility function the range of achievable input rates without any compensation is $0\leq \lambda \leq \lambda_{CF,max}= \min(\lambda_{CF}^e(0;c),\Lambda)$, and $\lambda_{CF}^e(0;c) = \mu-\frac{rc}{1-e^{-rR}}$. In general when a compensation policy is used, the corresponding range increases, and
one of the goals in the numerical analysis is to compare achievable ranges with and without compensation. In the following we refer to input rates that are below $\lambda_{CF,max}$ and thus can be achieved without compensation as low, and to rates above $\lambda_{CF,max}$, which can be achieved only when appropriate compensation is offered, as high. 

In the following computational experiments we use a base case of parameter values  $R=15,c=8,\mu=12, r=0.5, p=10, l=4.5, p=10, d=0.5$ and in each case let one or more of the parameters to vary. Since the market size $\Lambda$ only affects the results by restricting the achievable ranges, we assume $\Lambda > \mu$.
Under these values the maximum achievable rate without compensation is equal to 
$\lambda_{CF,max}=7.99$ and the optimal rate under risk neutrality is equal to 
$\lambda^*=9.47$.

In the first set of computational experiments we consider the load-control problem when one policy parameter is kept fixed at a time. For each case we derive the corresponding pricing curves and allowed intervals for the remaining two free parameters under various values of $\lambda$.

\begin{figure}
\includegraphics[width=8.7cm,height=7.0cm]{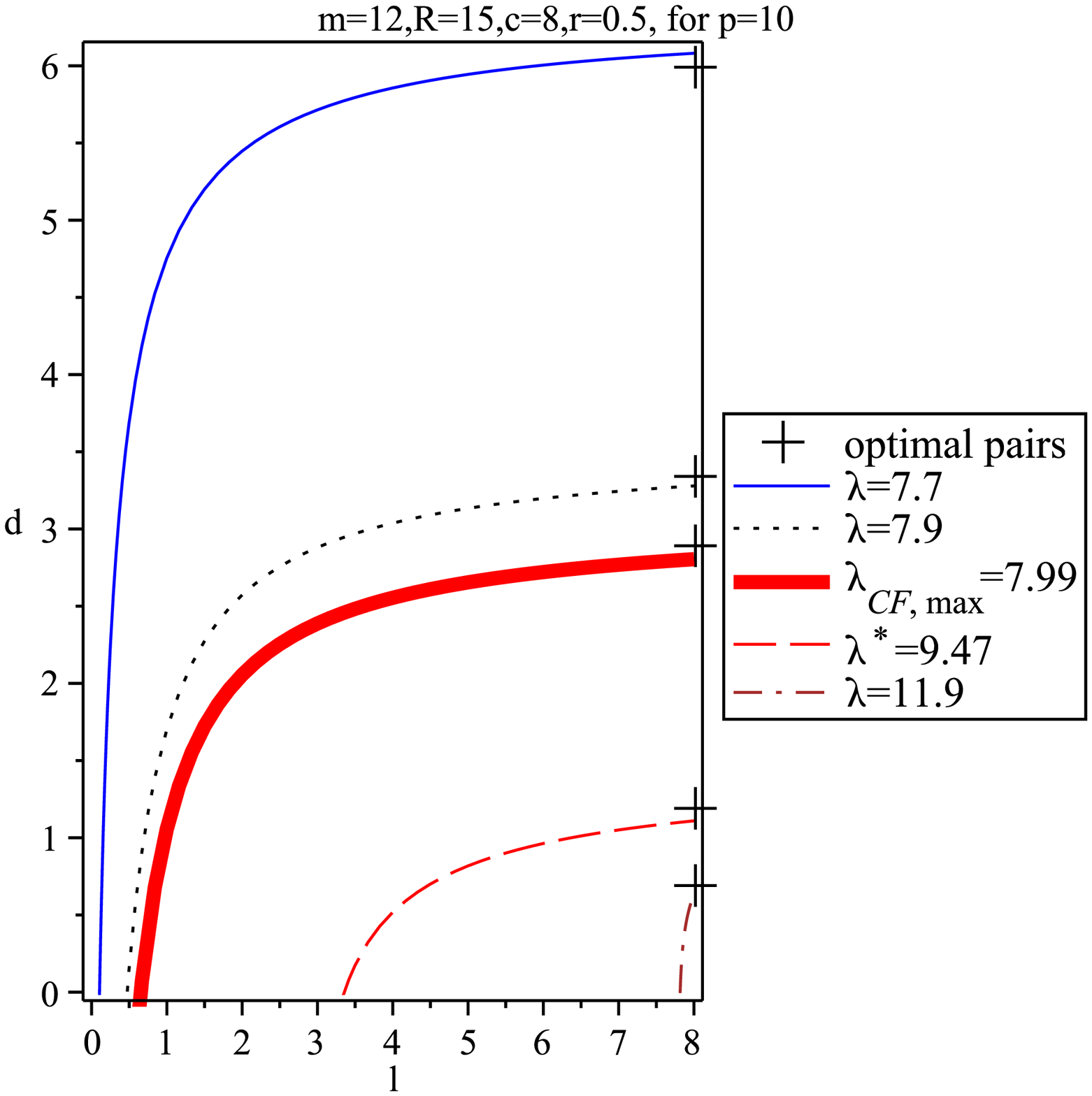}
\includegraphics[width=7.5cm,height=7.0cm]{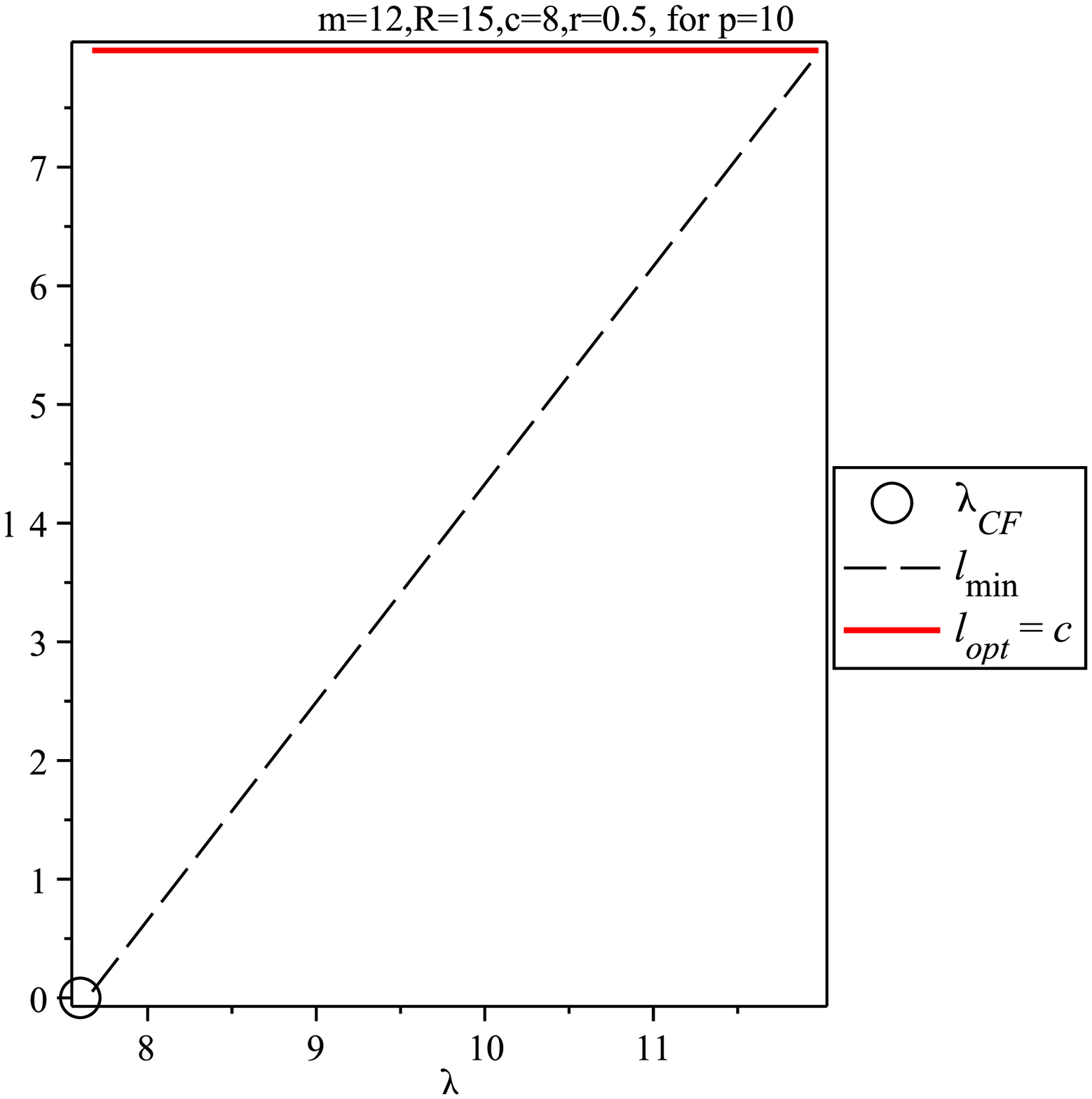}
\caption{$p$ fixed}
\label{pfixed}
\end{figure}

Figure \ref{pfixed} refers to the case when the entrance fee is fixed at $p=10$. The left graph shows the $(d,l)$-pricing curves under various achievable input rates. On each curve the profit maximizing pair is indicated by the cross symbol. 
The right graph shows the range of compensation rates $l$ within which an input rate $\lambda$ is achievable. The profit maximizing value is at the highest value of the range. 

We first observe that under no compensation the only achievable input rate is $\lambda_{CF}^e(p;c) = 7.64$, whereas when a compensation policy is used the achievable range is expanded to the interval $[7.64, 12)$. Values lower than $\lambda_{CF}^e(p;c)$ are not achievable, since the compensation policy can only increase the input rate. 
 
Moreover, the pricing curves are increasing as expected, since when the lead-time is increased the provider must also increase the compensation in order to keep the input rate constant. Similarly, when $l$ is fixed, the required lead-time is decreasing as the desired input rate increases.
 
Finally, on each pricing curve the profit maximizing pair corresponds to the highest value of the compensation rate, i.e., $l=c$. 
Recall from Theorem \ref{eoptimal} that when the provider is completely flexible in setting the policy parameters, he or she prefers to set the entrance fee very close to $R$ and fully compensate all customer delay for any desired value of $\lambda$. Here we see that when the entrance fee is kept fixed at a value below $R$, then the provider prefers a policy that fully compensates all excess delay above an appropriate positive lead-time, rather than partially compensating the delay from the first minute.

\begin{figure}
\includegraphics[width=8.7cm,height=7.0cm]{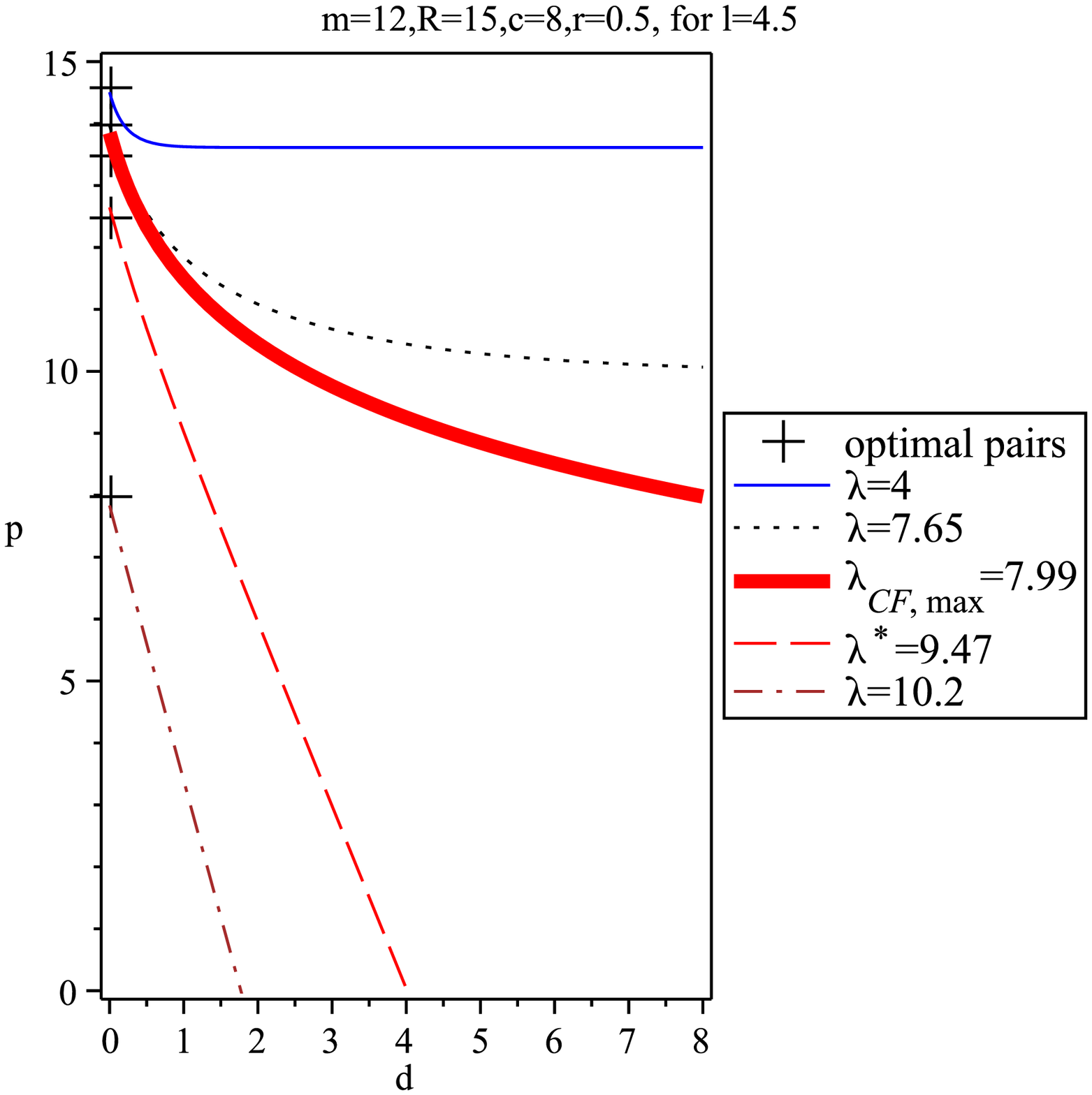}
\includegraphics[width=7.5cm,height=7.0cm]{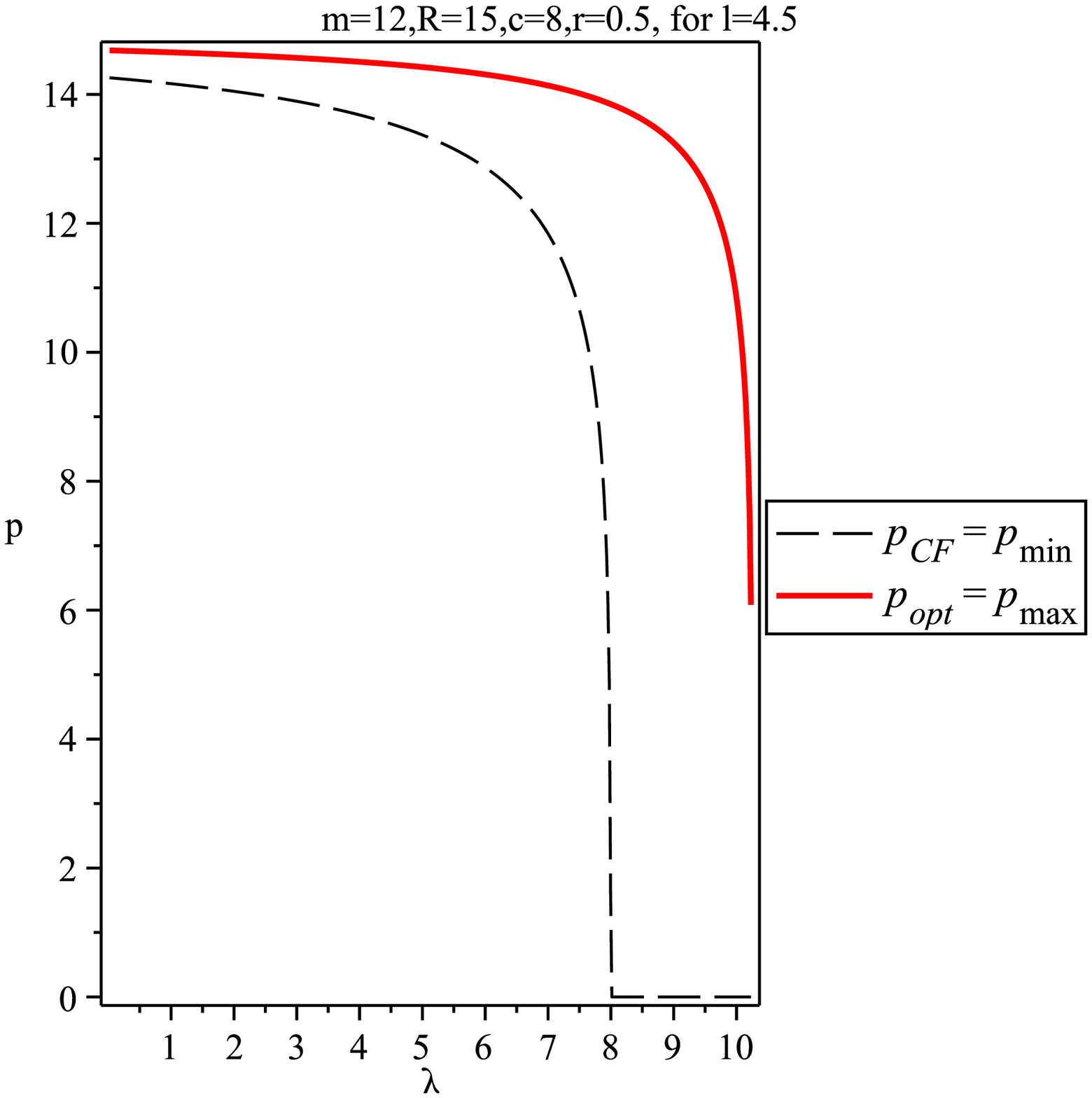}
\caption{$l$ fixed}
\label{lfixed}
\end{figure}

Figure \ref{lfixed} refers to the case when the compensation rate is fixed at $l=4.5$. The left graph shows the $(p,d)$-pricing curves for various values of $\lambda$. The right graph shows the range of entrance fees within which each input rate is achievable. For $\lambda< \lambda_{CF,max}$, the lowest value of the range corresponds to the price $p_{CF}(\lambda;c)$  that results in input rate $\lambda$ under no compensation. The highest value corresponds to the profit maximizing price. 
For the given value of $l$ the achievable range of input rates from Theorem \ref{onefixedachievable} is the interval $[0, 10.24]$, compared to the interval $[0, 7.99]$ under no compensation.

In this case the pricing curves are decreasing for each value of $\lambda$, i.e., the provider must set a lower entrance fee as the lead-time for compensation is increased. Furthermore, for low values of the input rate $\lambda< \lambda_{CF,max}$ the provider may choose effectively not to offer compensation by setting $d$ to an arbitrarily high value and still set a positive value of $p$. On the other hand, if the desired input rate is high, not offering compensation is not an option, since the entrance fee becomes zero for a finite value of $d$. 

On each pricing curve the profit maximizing pair corresponds to $d=0$ and  the highest possible value of the entrance fee. This implies that for any desired input rate, the provider prefers to compensate at the given rate $l$ from the first minute of delay and be able to set a high entrance fee rather than lower the fee and increase the lead-time. 

Figure \ref{dfixed} refers to the case when the lead-time is fixed at $d=0.5$. The left graph shows the $(p,l)$-pricing curves under various achievable input rates. The right graph shows the range of entrance fees within which each input rate is achievable. As in the previous case, for $\lambda< \lambda_{CF,max}$, the lowest value of the range corresponds to the price $p_{CF}(\lambda;c)$ that results in input rate $\lambda$ under no compensation. Again, the highest value corresponds to the profit maximizing price. For the given value of $d$ the achievable range of input rates from Theorem \ref{onefixedachievable} is the interval $[0, 12)$, which is extended compared to the interval $[0, 7.99]$ under no compensation.

Additionally, the pricing curves are increasing for each value of $\lambda$, i.e., the provider can set a higher entrance fee as the compensation rate is increased. Also, for low values of the input rate $\lambda< \lambda_{CF,max}$ the provider has the option to offer no compensation by setting a minimum positive value of $p$. However, if the desired input rate is high, he or she always has to offer a positive compensation even if he or she offers a zero entrance fee.

Finally, on each pricing curve the profit maximizing pair corresponds to $l=c$ and the highest possible value of the entrance fee. This implies that for any desired input rate, the provider prefers to fully compensate all excess delay above a given lead-time and be able to set a high entrance fee rather than lower the fee and decrease the compensation rate.
Furthermore, the optimal entrance fee $p$ is decreasing in $\lambda$, in contrast to the case where all parameters are flexible and the optimal entrance fee is essentially equal to the service reward $R$. 

\begin{figure}
\includegraphics[width=8.7cm,height=7.0cm]{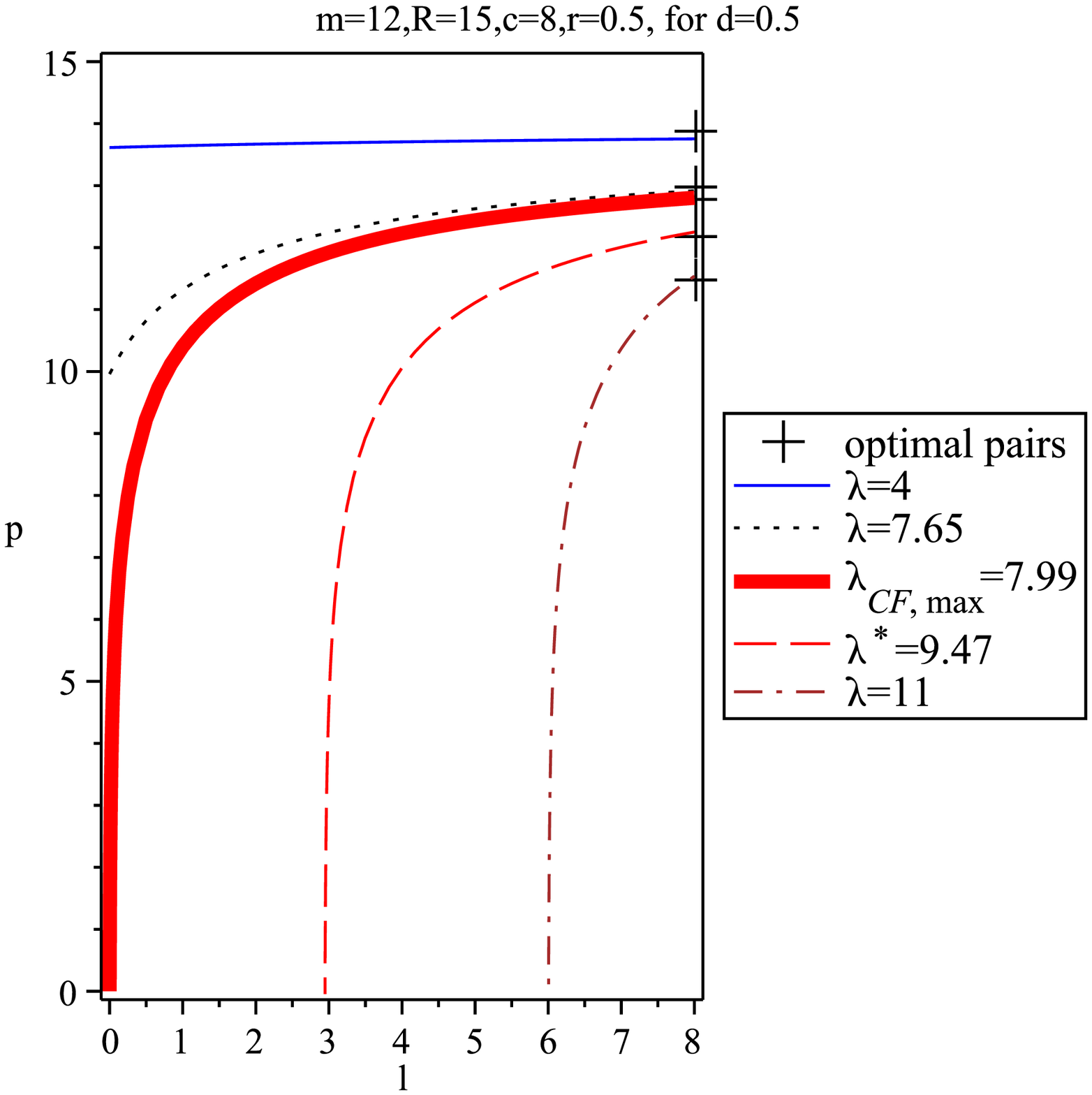}
\includegraphics[width=7.5cm,height=7.0cm]{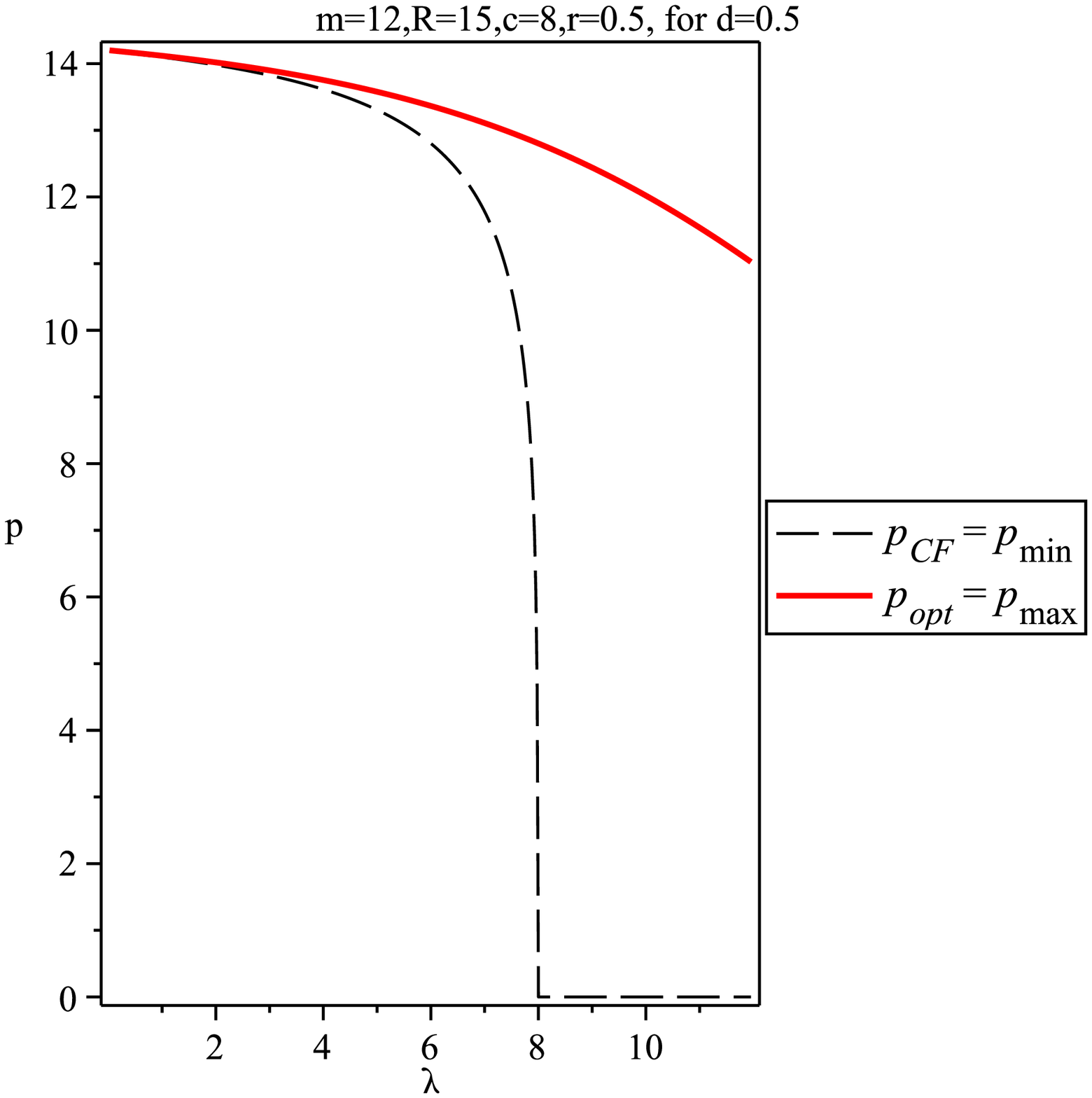}
\caption{$d$ fixed}
\label{dfixed}
\end{figure}

A question that has not been addressed so far is related to the actual profit that the provider may achieve and in particular to what extent the compensation policy may help recover the profit reduction due to customer risk aversion. Figure \ref{optimalprofits} presents the optimal profit as a function  of the input rate under no compensation, under full policy flexibility, and under each of the above three cases with one parameter fixed. Several interesting insights can be derived from these graphs. The lowest curve (CF) corresponds to the optimal profit under no delay compensation, and the highest curve (H) to the optimal profit $H(\lambda)$ that that can be achieved under customer risk neutrality and, as we saw in the previous section, can be approached arbitrarily close when the provider has full flexibility on policy parameters. The peak values of these two curves differ by approximately 30\%, which shows that under this value of the risk parameter $r$ the effect of risk aversion on the profitability is significant. 

The three intermediate profit curves show that by using a compensation policy the provider can recover a substantial part of this profit loss in general. For the given set of parameter values, we observe that keeping the lead-time $d$ fixed and having flexibility on $p,l$ is the most beneficial case, both in terms of the range of $\lambda$ that can be achieved and in terms of the profit function. Note that the value of $d=0.5$ that has been used is relatively very high, since it corresponds to six times the mean service time. This shows that even by offering a low-impact compensation policy, under which customers are rarely compensated, the provider can significantly improve his/her profit and approach the risk-neutral curve closely. 

This discussion confirms the observations reached in the three separate cases of fixed parameters above, namely, that the most instrumental feature of the compensation policy is that it allows the provider to offer the risk-averse customers a hedge against high delays and at the same time increase the entrance fee close to the levels that are optimal under risk-neutrality. This was shown analytically in Section \ref{profitmax} under full policy flexibility, and it still holds when the policy is constrained in some way. Although the discussion here refers to a fixed value of $r$, the insights seem to be robust under several other parameter combinations. 

Returning to Section \ref{profitmax}, it was shown that under full policy flexibility, the optimal profit function can be approached by providing full compensation from the first minute and setting the entrance fee arbitrarily close to the service reward $R$. From the discussion above it follows that this is not the only possibility. In fact, an alternative approach for the provider may be to offer full compensation with a relatively high lead-time, and set the entrance fee at the more modest level of the maximizing price under no compensation.  

\begin{figure}
\includegraphics[width=1.1\textwidth]{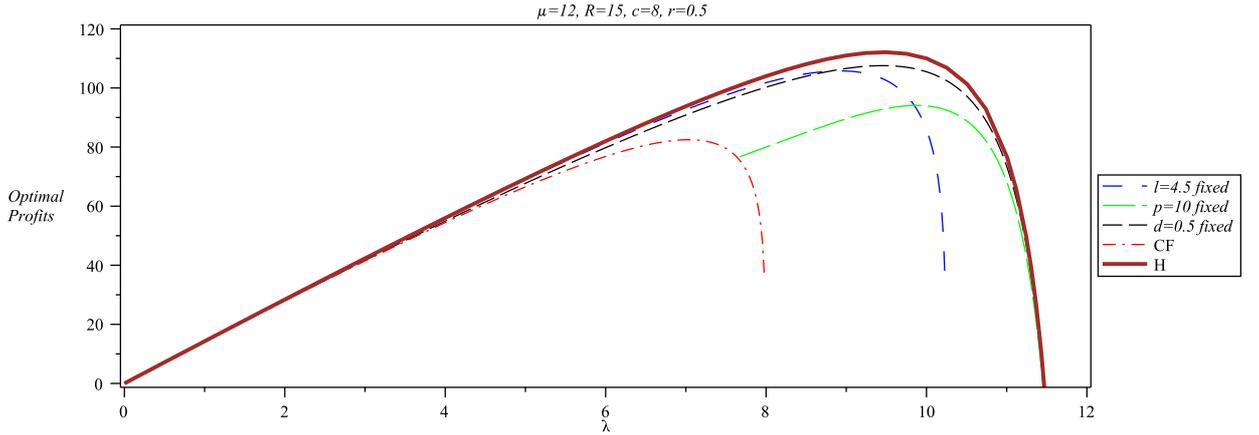}
\caption{Optimal profit function}
\label{optimalprofits}
\end{figure}

In the second set of computational experiments we focus on the effect of the degree of risk aversion on the range of achievable input rates and the corresponding policy curves. For clarity of the exposition we consider the cases of two fixed policy parameters. For each case, Figure \ref{riskcurves} shows how the third policy parameter varies as  a function of $r$ for various values of $\lambda$. 

A first observation from these graphs is that when the free policy parameter is $d$ or $p$, it is decreasing in $r$ for any value of $\lambda$, while when $l$ is the free parameter it is increasing in $r$. These are all intuitive, since increasing $r$ corresponds to higher risk aversion and thus the provider must move to a more customer beneficial direction in order to keep the same input rate. 
Furthermore, in the first two cases, as $r$ increases, higher values of the input rate gradually stop being achievable. On the other hand, when $l$ is flexible, high values of $\lambda$ are always achievable by setting the compensation rate closer to $c$. The resulting policy guideline is that if high input rates are desirable under high risk aversion and the policy parameters are constrained, then the relaxation must be towards a higher compensation rate for excess delay rather than a lower entrance fee or a lower lead-time.

\begin{figure}
\includegraphics[width=5.4cm,height=5.4cm]{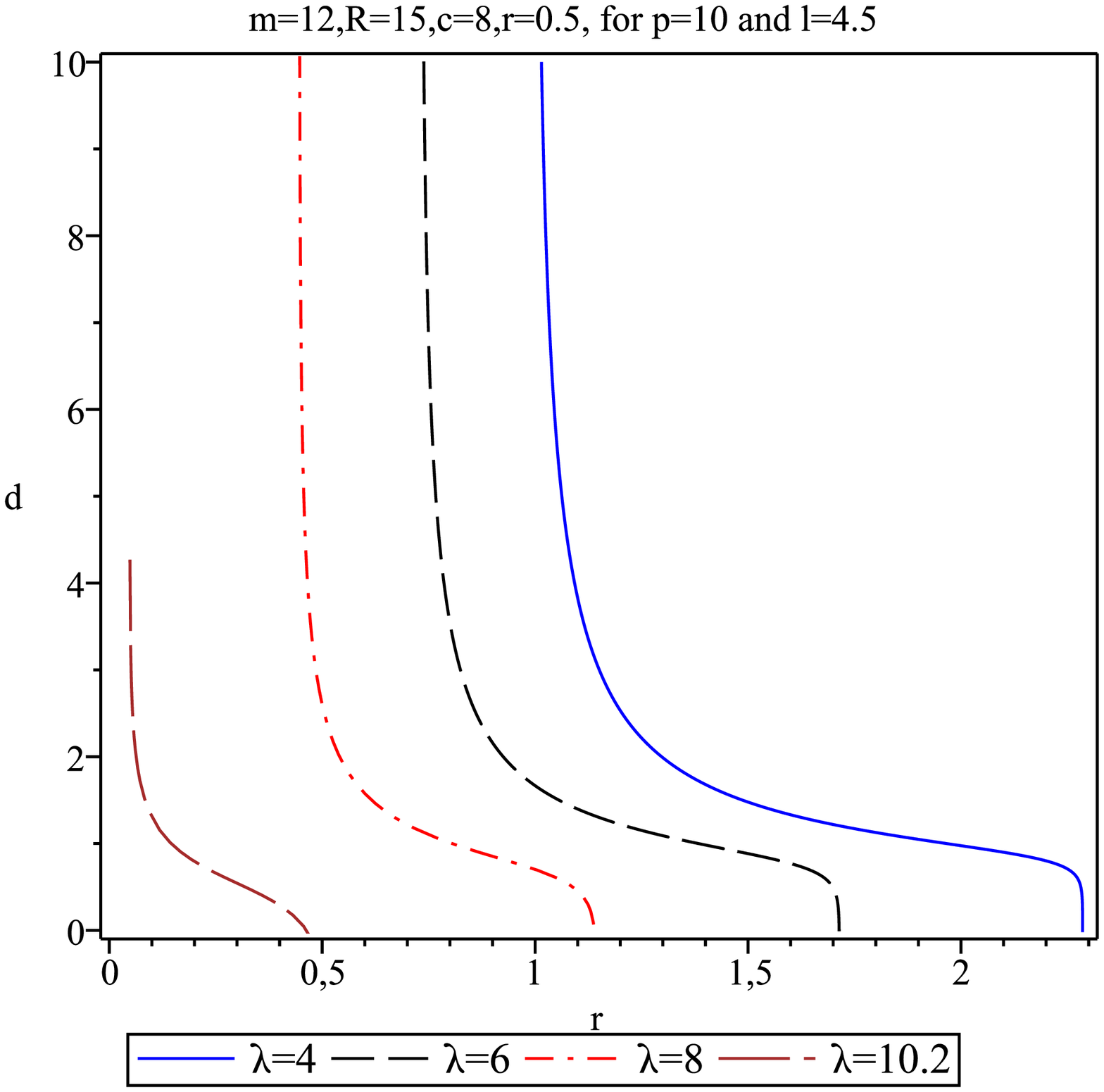}
\includegraphics[width=5.4cm,height=5.4cm]{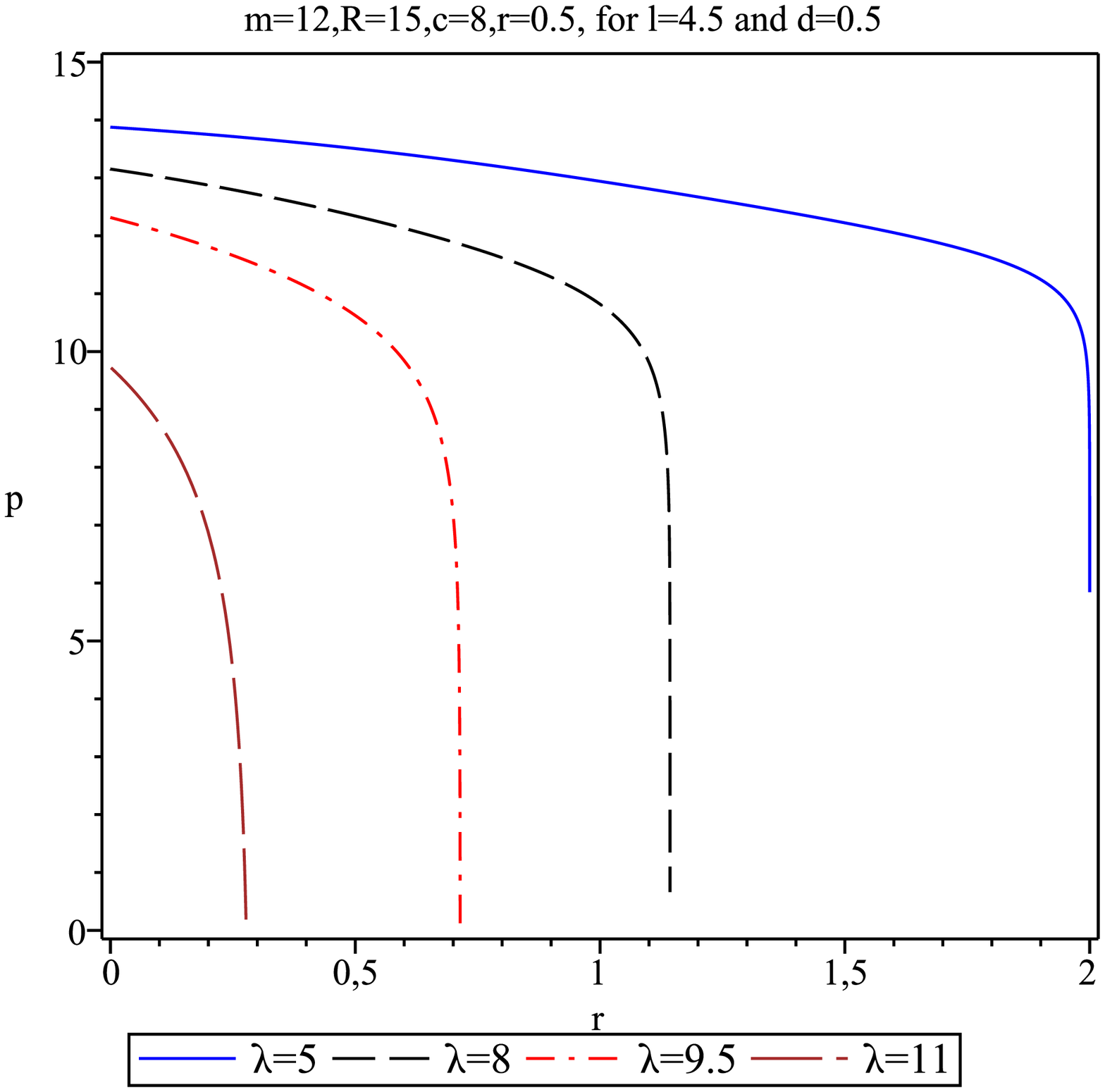}
\includegraphics[width=5.4cm,height=5.4cm]{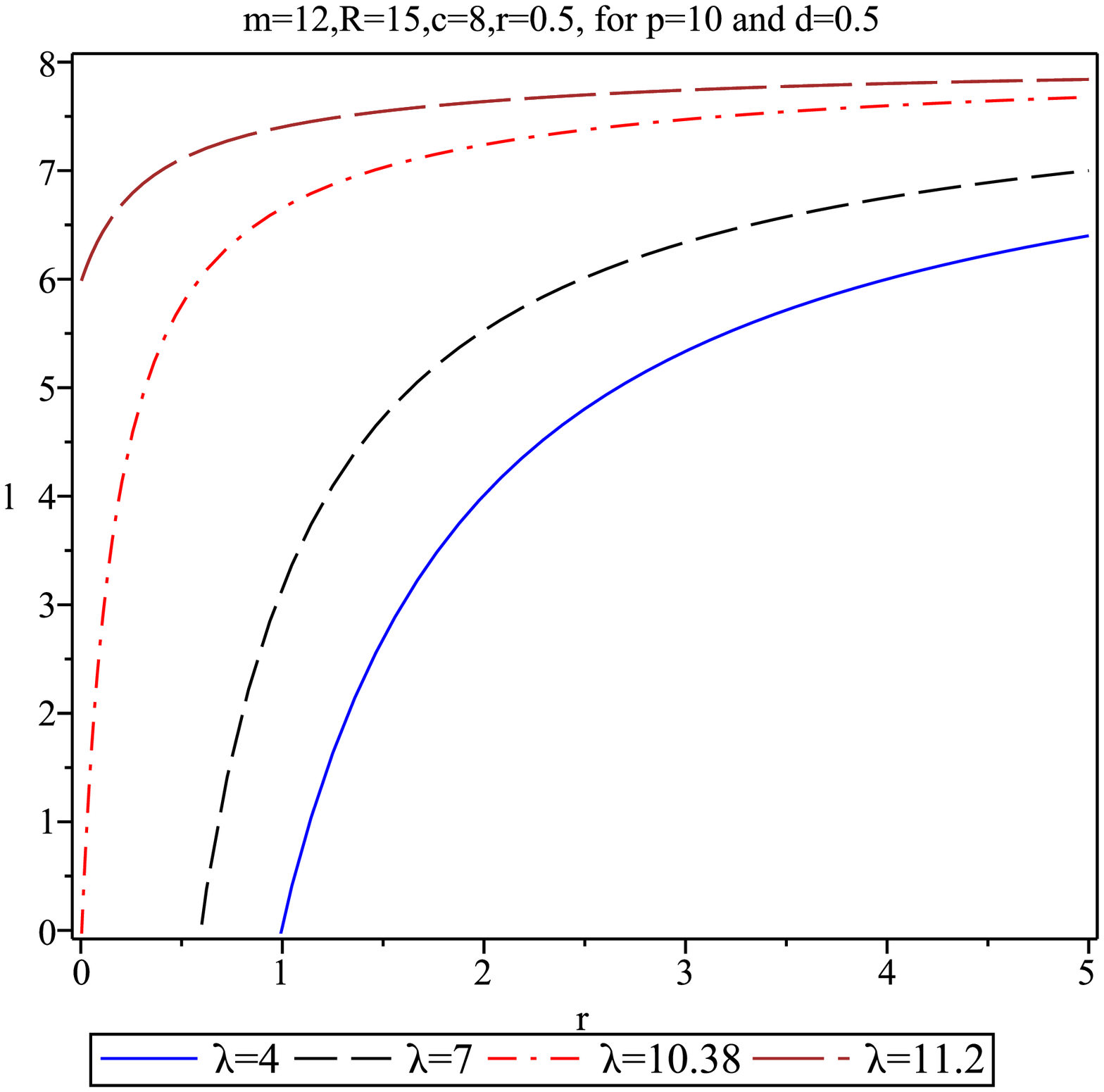}
\caption{Risk aversion variation}
\label{riskcurves}
\end{figure}

\section{Conclusions and Extensions}\label{conclusions}
In this paper we develop customer equilibrium strategies and profit maximizing policies for an unobservable
make-to-order/service production system with risk averse customers, where the service provider, in addition to the entrance fee, sets a common lead-time and linear compensation for the excess delay above the lead-time quote. 
The risk aversion is modeled by a concave utility function of the customer net benefit.
We analyze the effect of customer risk aversion and the compensation policy on the equilibrium strategies and  explore the provider's flexibility regarding the achievable input rates  under various constraints on the pricing/compensation policy.
We prove that when the provider has full flexibility in the pricing/compensation policy, he or she can mitigate the adverse effects of risk aversion almost entirely, by employing a policy that compensates fully for the excess delay and sets the entrance fee close to the customer service benefit. However in this case there is no strictly optimal policy.
In numerical experiments we construct  pricing curves which show the provider's flexibility in  inducing specific input rates, as well as the sensitivity of the provider's policies in the degree of risk aversion. 

This work could be extended in several ways. A natural extension is to consider the observable system, where customers are informed about the system congestion before making their decision and the provider may employ state-dependent quotation policies. This work is currently under progress. In other directions, one may consider the case of a make-to-order system where a physical component is required to complete the service of a customer, and examine the interaction between pricing/compensation and stocking policies.
Finally, it would be interesting to develop a competitive model with more than one providers offering the same service and explore how the equilibrium pricing/compensation policies differ from those formed under a single entrance free case,  as well as to what extent the compensation mechanism now  helps alleviate the risk aversion effects.

\end{document}